\documentclass[english,11pt]{article}
\usepackage[T1]{fontenc}
\usepackage[utf8]{inputenc}
\usepackage{refstyle}
\usepackage{mathrsfs}
\usepackage{amsmath}
\usepackage{amsthm}
\usepackage{amssymb}
\usepackage{graphicx}

\makeatletter

\AtBeginDocument{\providecommand\propref[1]{\ref{prop:#1}}}
\AtBeginDocument{\providecommand\thmref[1]{\ref{thm:#1}}}
\AtBeginDocument{\providecommand\secref[1]{\ref{sec:#1}}}
\AtBeginDocument{\providecommand\asmref[1]{\ref{asm:#1}}}
\AtBeginDocument{\providecommand\lemref[1]{\ref{lem:#1}}}
\AtBeginDocument{\providecommand\figref[1]{\ref{fig:#1}}}
\RS@ifundefined{subsecref}
  {\newref{subsec}{name = \RSsectxt}}
  {}
\RS@ifundefined{thmref}
  {\def\RSthmtxt{theorem~}\newref{thm}{name = \RSthmtxt}}
  {}
\RS@ifundefined{lemref}
  {\def\RSlemtxt{lemma~}\newref{lem}{name = \RSlemtxt}}
  {}

\theoremstyle{plain}
\newtheorem{thm}{\protect\theoremname}
\theoremstyle{plain}
\newtheorem{prop}[thm]{\protect\propositionname}
\theoremstyle{remark}
\newtheorem{rem}[thm]{\protect\remarkname}
\theoremstyle{plain}
\newtheorem{lem}[thm]{\protect\lemmaname}

\usepackage{url}
\usepackage{xcolor}
\usepackage[margin=1in]{geometry}
\usepackage[normalem]{ulem}

\newtheorem{assume}{Assumption}

\makeatother

\usepackage{babel}
\providecommand{\lemmaname}{Lemma}
\providecommand{\propositionname}{Proposition}
\providecommand{\remarkname}{Remark}
\providecommand{\theoremname}{Theorem}

\begin{document}
\title{Nonnegative Low-rank Matrix Recovery Can Have Spurious Local Minima
\thanks{Financial support for this work was provided by NSF CAREER Award ECCS-2047462
and ONR Award N00014-24-1-2671.}}
\author{Richard Y. Zhang\\
Dept. of Electrical and Computer Engineering\\
University of Illinois at Urbana-Champaign\\
306 N Wright St, Urbana, IL 61801\\
\url{ryz@illinois.edu}}
\maketitle
\begin{abstract}
Low-rank matrix recovery is well-known to exhibit \emph{benign nonconvexity}
under the restricted isometry property (RIP): every second-order critical
point is globally optimal, so local methods provably recover the ground
truth. Motivated by the strong empirical performance of projected
gradient methods for nonnegative low-rank recovery problems, we investigate
whether this benign geometry persists when the factor matrices are
constrained to be elementwise nonnegative. In the simple setting of
a rank-1 nonnegative ground truth, we confirm that benign nonconvexity
holds in the fully-observed case with RIP constant $\delta=0$. This
benign nonconvexity, however, is unstable. It fails to extend to the
partially-observed case with any arbitrarily small RIP constant $\delta>0$,
and to higher-rank ground truths $r^{\star}>1$, regardless of how
much the search rank $r\ge r^{\star}$ is overparameterized. Together,
these results undermine the standard stability-based explanation for
the empirical success of nonconvex methods and suggest that fundamentally
different tools are needed to analyze nonnegative low-rank recovery.
\end{abstract}
\global\long\def\A{\mathcal{A}}%
\global\long\def\AT{\A^{*}}%
\global\long\def\B{\mathcal{B}}%
\global\long\def\inner#1#2{\left\langle #1,#2\right\rangle }%
\global\long\def\R{\mathbb{R}}%
\global\long\def\S{\mathbb{S}}%
\global\long\def\Cone{\mathscr{C}}%
\global\long\def\rank{\operatorname{rank}}%
\global\long\def\svec{\operatorname{svec}}%
\global\long\def\vec{\operatorname{vec}}%
\global\long\def\diag{\operatorname{diag}}%
\global\long\def\cond{\operatorname{cond}}%
\global\long\def\tr{\operatorname{tr}}%
\global\long\def\sm{\operatorname{sum}}%
\global\long\def\rip{\operatorname{RIP}}%
\global\long\def\e{\mathbf{e}}%
\global\long\def\one{\mathbf{1}}%
\global\long\def\half{{\textstyle \frac{1}{2}}}%

\section{Introduction}

Low-rank matrix recovery~\cite{recht2010guaranteed}---recovering
an unknown low-rank matrix from partial observations---is a basic
problem in modern machine learning and signal processing. A common
approach is to represent the matrix in factored form, such as $UU^{T}$
or $LR^{T}$, and to fit the factored matrix to data using simple
gradient-based methods. Although the resulting optimization problem
is nonconvex and NP-hard in general, these simple methods often work
surprisingly well in practice. A standard explanation is the manifestation
of \emph{benign nonconvexity in tractable regimes}. For example, under
the Restricted Isometry Property (RIP) with a sufficiently small constant
$\delta$, the optimization landscape becomes devoid of spurious
local minima, so local optimization always converges to the global
minimum~\cite{bhojanapalli2016global,ge2017nospurious,zhu2018global}.

In clustering~\cite{xu2003document,zhuang2024statistically} and
image analysis~\cite{lee1999learning} applications, however, the
factors are required to be elementwise nonnegative $U\ge0$ (or $L,R\ge0$).
Imposing nonnegativity is algorithmically straightforward via a projection
step $U_{ij}\gets\max\{U_{ij},0\}$. Empirically, the resulting projected
gradient methods often behave much like their unconstrained counterparts:
once the underlying problem enters a tractable regime, nonconvex optimization
appears to succeed reliably. In particular, a recent nonnegative low-rank
method for $K$-means clustering~\cite{zhuang2024statistically}
recovers planted Gaussian mixture models from essentially any starting
point as soon as the clusters are well separated enough for recovery
to be polynomial-time solvable by convex relaxation~\cite{chen2021cutoff}.

This empirical parallel---a sharp transition in optimization behavior
coinciding with the onset of tractability---naturally raises the
question: \emph{Does the benign landscape of RIP low-rank matrix recovery
persist once nonnegativity constraints are imposed?} 

This paper performs a systematic investigation and obtains a largely
negative answer. While nonnegative low-rank recovery does exhibit
benign nonconvexity in a single base case, this fails to persist under
mild extensions that are well-known to be benign in unconstrained
recovery:
\begin{itemize}
\item Benign landscape holds with a rank $r^{\star}=1$ ground truth in
the fully observed $\delta=0$ regime, when all $n_{1}n_{2}$ elements
of the $n_{1}\times n_{2}$ rank-1 matrix are directly measured (\propref{nospu}).
\item Benign landscape does not extend to the \emph{partially observed}
case $\delta>0$, no matter how small the RIP constant (\thmref{main}).
As shown in Figure~\ref{fig:counter}, spurious local minima can
be induced at the boundary for any $\delta\to0^{+}$, even though
the unconstrained landscape becomes benign once $\delta<1/2$~\cite{zhang2019sharp,zhang2024improved}.
\item Benign landscape does not extend to \emph{higher-rank} ground truths
$r^{\star}>1$, even when the specific instance of nonnegative recovery
is trivial and easily solved by inspection (\propref{spu2}). 
\item Spurious local minima can persist in the \emph{overparameterized}
regime $r>r^{\star}$, regardless of how large the search rank $r$
is set with respect to the true rank $r^{\star}$ (\thmref{main}
and \propref{spu2}). In contrast, unconstrained recovery enjoys benign
landscape as soon as the overparameterization ratio exceeds $r/r^{\star}>[\delta/(1-\delta)]^{2},$
where $\delta$ is the RIP constant at rank $r+r^{\star}$~\cite{zhang2024improved,zhang2025sharp}.
\end{itemize}
As we explain next, these negative results are significant because
they undermine the stability argument that is widely used to explain
the empirical success of nonconvex optimization algorithms. 

\begin{figure}
\includegraphics[width=0.5\textwidth]{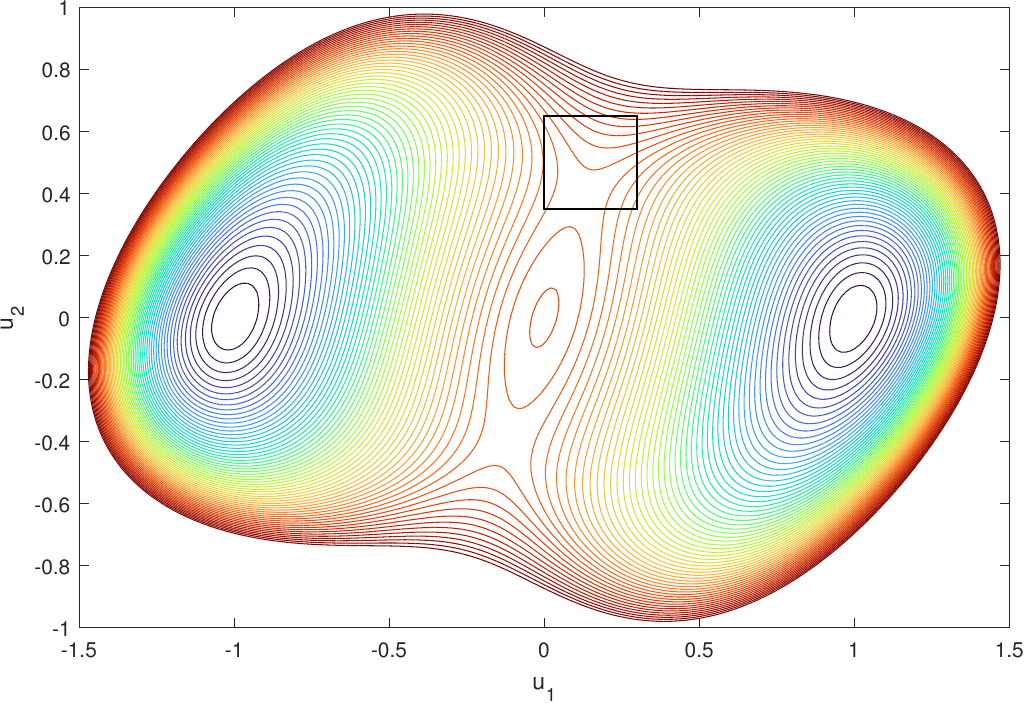}\includegraphics[width=0.5\textwidth]{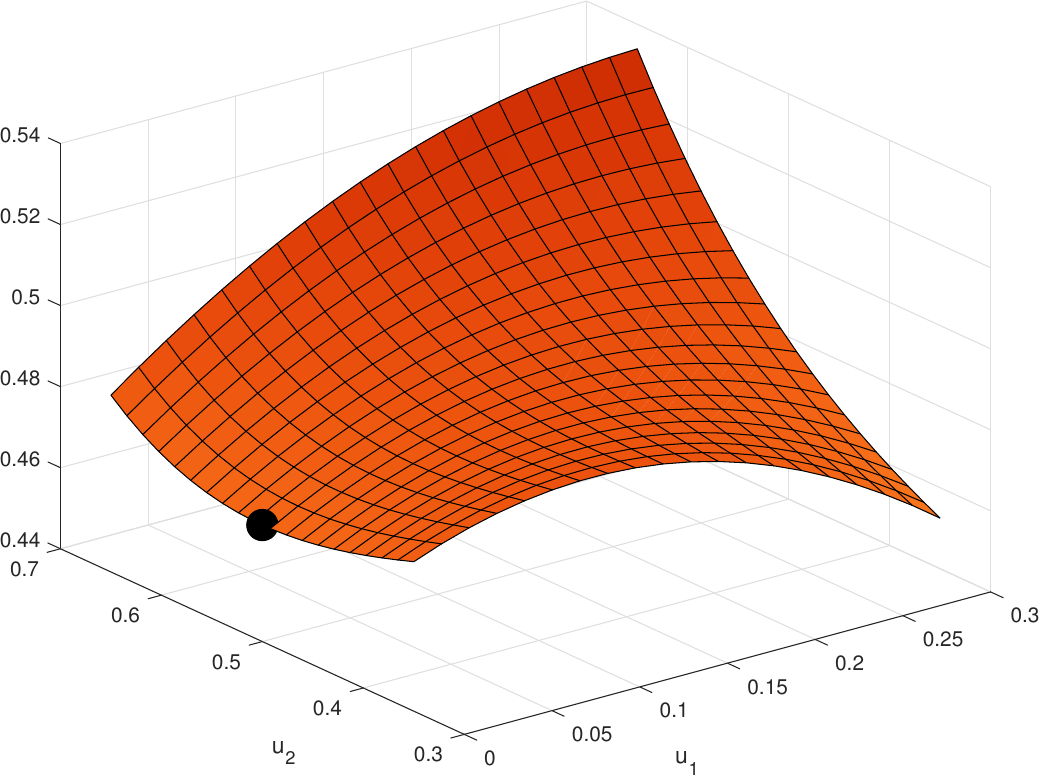}

\caption{\textbf{Rank-1 counterexample with partial observations.} \emph{Left:}
Contour plot of $f(u)=\protect\half\|\mathcal{A}(uu^{T}-u_{\star}u_{\star}^{T})\|^{2}$
taken from \thmref{main} with $n=2$ and $r=1$. The landscape is
benign over the entire domain $\protect\R^{2}$, but the partial observations
$\delta>0$ induce a slight off-axis rotation. \emph{Right:} Restricting
$u$ to the positive orthant $\mathbb{R}_{+}^{2}$ causes the boundary
point $u_{0}=(0,0.5)$ to become a strict local minimizer. }\label{fig:counter}
\end{figure}

\subsection{Benign landscape via stability argument}

A large body of work over the past ten years has established benign
nonconvexity across a wide range of nonconvex optimization problems.
These results span many instances of unconstrained low-rank matrix
recovery, including matrix sensing~\cite{bhojanapalli2016global},
matrix completion~\cite{ge2016matrix}, robust PCA~\cite{ge2017nospurious},
and phase retrieval~\cite{sun2016guaranteed}, as well as equality-constrained
versions, such as phase and orthogonal group synchronization~\cite{bandeira2016low}.
Despite the diversity of settings, all these works share a common
proof strategy. One first identifies a base case ``prototype'' setting
in which the optimization landscape can be analyzed exactly and shown
to be benign. On then proves that the landscape remains stable under
small changes, such as partial observation, imbalance or asymmetry,
higher-rank ground truth, overparameterization, or mild noise.

In unconstrained low-rank matrix recovery, the fully observed symmetric
rank-1 problem serves as the most basic prototype. Consider the unconstrained
symmetric low-rank matrix recovery cost 
\begin{equation}
f:\R^{n\times r}\to\R,\quad U\mapsto\half\|\A(UU^{T}-U_{\star}U_{\star}^{T})\|^{2},\label{eq:fdef}
\end{equation}
where $U_{\star}\in\R^{n\times r^{\star}}$ is unknown, the search
rank $r\ge r^{\star}$, and the linear operator $\A:\R^{n\times n}\to\R^{m}$
is assumed to satisfy the \emph{restricted isometry property}\footnote{Throughout this paper, $\|\cdot\|$ denotes the matrix Euclidean (Frobenius)
norm. } 
\begin{equation}
(1-\delta)\|M\|^{2}\le\|\A(M)\|^{2}\le(1+\delta)\|M\|^{2}\quad\text{ for all }\rank(M)\le r+r^{\star}.\tag{RIP}\label{eq:rip}
\end{equation}
When $\delta=0$ and $r=r^{\star}=1$, the benign geometry of $f$
is easily established by verifying that the usual second-order necessary
condition is also sufficient for global optimality. This property
turns out to be structurally stable; one can show that introducing
small partial observation $\delta>0$, increasing the true rank $r^{\star}\ge1$,
or overparameterizing the search rank $r\ge r^{\star}$ does not create
new second-order critical points~\cite{bhojanapalli2016global,ge2017nospurious}.

The same mechanism underlies extensions to asymmetric factorizations.
The key idea is to view the $n_{1}\times n_{2}$ matrix $LR^{T}$
as the top-right off-diagonal block of the $n\times n$ matrix $UU^{T}$
of the same rank, where $n=n_{1}+n_{2}$ and $U=[L;R]$ is the column-wise
concatenation of the factors. One then analyzes the asymmetric objective
\begin{equation}
g:\R^{n\times r}\to\R,\quad[L;R]\mapsto\half\|\A(LR^{T}-L_{\star}R_{\star}^{T})\|^{2},\label{eq:gdef}
\end{equation}
augmented with the balancing regularizer
\begin{equation}
h:\R^{n\times r}\to\R,\quad[L;R]\mapsto\half\|L^{T}L-R^{T}R\|^{2},\label{eq:hdef}
\end{equation}
which removes the scaling ambiguity $LR^{T}=(LG)(RG^{-T})^{T}$. This
reduction expresses the asymmetric problem as a controlled perturbation
of the symmetric prototype. The stability of benign landscape in the
symmetric case then allows the same property to extend to the asymmetric
setting, up to constant-factor adjustments~\cite{park2017non}. 

Finally, the same stability argument extends to more challenging measurement
models that only approximately satisfy (\ref{eq:rip}). In the matrix
completion, for example, the condition $\|\A(M)\|^{2}\approx\|M\|^{2}$
holds only for low-rank matrices that are sufficiently dense. Ge et
al.~\cite{ge2016matrix} established benign nonconvexity in this
setting by augmenting $f$ with a regularizer that enforces density,
thereby reducing the problem back to the symmetric RIP prototype.
Subsequent work shows that benignness remains stable when the ground
truth matrices $U_{\star},L_{\star},R_{\star}$ are only approximately
low-rank, and when algorithms converge only to approximate critical
points \cite{ge2017nospurious}. 

\subsection{Breakdown of landscape stability under nonnegativity}

Unfortunately, the same stability argument does not carry through
to establish benign nonconvexity under nonnegativity. We consider
both the symmetric variant 
\begin{equation}
\min_{U\in\R^{n\times r}}\quad f(U)\quad\text{ subject to }\quad U\ge0,\label{eq:nmf-sym}
\end{equation}
and the regularized asymmetric variant (for some $\lambda\ge0$)
\begin{equation}
\min_{L\in\R^{n_{1}\times r},R\in\R^{n_{2}\times r}}\quad g([L;R])+\lambda\;h([L;R])\quad\text{ subject to }\quad L\ge0,\quad R\ge0,\label{eq:nmf}
\end{equation}
where $f,g,h$ are defined as in the unconstrained case in (\ref{eq:fdef}),
(\ref{eq:gdef}), and (\ref{eq:hdef}), but with respect to $U_{\star}\in\R_{+}^{n\times r^{\star}},$
$L_{\star}\in\R_{+}^{n_{1}\times r^{\star}},$ and $R_{\star}\in\R_{+}^{n_{2}\times r^{\star}}$
that are further assumed to be elementwise nonnegative. 

In this nonnegative setting, the fully observed $\delta=0$ problem
is NP-hard in general, as it encompasses completely positive factorization~\cite{dickinson2014computational}
and nonnegative matrix factorization~\cite{vavasis2010complexity}.
At the same time, many instances are computationally trivial. For
example, $r^{\star}=1$ reduces to computing a top eigenvector (or
singular vector), and $r^{\star}=2$ reduces to a planar rotation~\cite{kalofolias2012computing}.
These instances also remain tractable for small $\delta>0$, since
one may first recover $UU^{T}$ or $LR^{T}$ by unconstrained methods
and then extract nonnegative factors.

We first confirm that the simplest fully observed rank-1 case is benign.
\begin{prop}
\label{prop:nospu}Let $\A$ satisfy (\ref{eq:rip}) with constant
$\delta=0$. For $r\ge r^{\star}=1$, the problems (\ref{eq:nmf-sym})
and (\ref{eq:nmf}) with $\lambda=1/4$ have no spurious local minima. 
\end{prop}

\begin{rem}
The choice $\lambda=1/4$ is essential; other values of $\lambda$
would break the proof. This weight matches the balancing regularizer
used in prior work on unconstrained low-rank recovery~\cite{park2017non,ge2017nospurious}.
\end{rem}

But in contrast to the unconstrained regime, this benign behavior
does not extend to any value of $\delta>0$, regardless of overparameterization.
Below, $\one_{r}$ denotes the length-$r$ vector of ones, and $I_{r}$
the $r\times r$ identity matrix.
\begin{thm}
\label{thm:main}For $r\ge r^{\star}\ge1$, choose $Q_{1}\in\R_{+}^{n\times r^{\star}}$
and $Q_{2}\in\R_{+}^{n\times r}$ so that the concatenation $\begin{bmatrix}Q_{1} & Q_{2}\end{bmatrix}$
consists of the first $r+r^{\star}\le n$ columns of an $n\times n$
permutation matrix. For $\epsilon>0$ and $0<\alpha<(r/r^{\star}+2\epsilon^{2}r)^{-1/4}$,
define $\A:\R^{n\times n}\to\R^{n\times n}$ to implement the following
map 
\[
X\mapsto X+\frac{\tr Q_{1}^{T}XQ_{1}}{r^{\star}}\begin{bmatrix}Q_{1} & Q_{2}\end{bmatrix}\begin{bmatrix}(-1+\sqrt{1-\delta^{2}})I_{r^{\star}} & -\epsilon\alpha^{2}\one_{r^{\star}}\one_{r}^{T}\\
-\epsilon\alpha^{2}\one_{r}\one_{r^{\star}}^{T} & \alpha^{2}I_{r}
\end{bmatrix}\begin{bmatrix}Q_{1} & Q_{2}\end{bmatrix}^{T}
\]
with $\delta=\alpha^{2}\sqrt{r/r^{\star}+2\epsilon^{2}r}$. Then,
$\A$ satisfies (\ref{eq:rip}) with constant $0<\delta<1$, but the
instance of (\ref{eq:nmf-sym}) with $U_{\star}=Q_{1}$ has a spurious
strict local minimum at $U_{0}=\alpha Q_{2}$. Moreover, the instance
of (\ref{eq:nmf}) with $L_{\star}=R_{\star}=Q_{1}$ has a spurious
second-order critical point at $L_{0}=R_{0}=\alpha Q_{2}$ for any
$\lambda\ge0$, which is a strict local minimum whenever $\lambda>0$.
\end{thm}

Thus, the landscape is not stable when moving from the fully observed
$\delta=0$ case to arbitrarily small partial observation $\delta>0$.
Our proof of \thmref{main} works by verifying the second-order sufficient
conditions for local optimality at the stated spurious points. In
the inequality-constrained setting, this conceptually simple task
is in fact NP-hard in general, as it amounts to copositive testing
on a submatrix of the Hessian~\cite{murty1987some}. Therefore, the
main difficulty of the proof is actually our careful construction
of $\A$ to allow copositive testing to be performed by hand; this
turns out to be particularly delicate in the asymmetric case. 

A second failure mode appears even before introducing partial observation.
In contrast to unconstrained recovery, where the fully observed case
is benign for all true ranks $r^{\star}$, nonnegative recovery can
already be non-benign at $\delta=0$ once $r^{\star}>1$. Even the
seemingly innocuous $r^{\star}=2$ fully observed setting can exhibit
spurious local minima.
\begin{thm}
\label{prop:spu2}Let $\A$ satisfy (\ref{eq:rip}) with constant
$\delta=0$. For $m>r\ge r^{\star}>1$, the instances of the problems
(\ref{eq:nmf-sym}) and (\ref{eq:nmf}) with $U_{\star}=L_{\star}=R_{\star}$
each have a spurious local minimum at $U=L=R=U_{0}$, where 
\[
U_{0}=\begin{bmatrix}\frac{1}{\sqrt{r}}\one_{m}\one_{r}^{T}\\
0
\end{bmatrix},\quad U_{\star}=\begin{bmatrix}\one_{m} & 0\\
0 & I_{r^{\star}-1}
\end{bmatrix}.
\]
\end{thm}

These examples are computationally trivial. They are instances of
the exact factorization problem (for $m>r\ge r^{\star}>1$) 
\[
\text{find }U\in\R_{+}^{n\times r}\ \text{such that}\ UU^{T}=M,\;M=\begin{bmatrix}\one_{m}\one_{m}^{T} & 0\\
0 & I_{r^{\star}-1}
\end{bmatrix}.
\]
The matrix $M$ is block diagonal with $r^{\star}$ rank-1 blocks.
The correct solution is evident by inspection, so the presence of
spurious local minima in this example cannot be attributed to computational
difficulty. Instead, it arises from the geometry of the nonnegative
factorization itself, exacerbated by extreme imbalance among the blocks.

This counterexample also admits a simple clustering interpretation.
The matrix $M$ can be viewed as the Gram matrix of $r^{\star}$ groups
of points in $\R^{r^{\star}}$, where each group consists of identical
points located at a distinct vertex of the standard simplex. Recovering
$U$ then amounts to assigning each point to one of these $r^{\star}$
clusters via a one-hot encoding. For example, the $r^{\star}=2$ instance
reduces to the problem of partitioning $m$ points at $(1,0)$ and
one point at $(0,1)$ into two clusters. 

\propref{spu2} then carries a cautionary message for clustering applications.
Even when clusters are perfectly separated, severe imbalances in cluster
size can induce spurious local minima. Empirical successes reported
in the clustering literature may therefore rely, implicitly or explicitly,
on additional normalization or balance conditions that preclude such
degeneracies. For example, the favorable behavior observed in~\cite{zhuang2024statistically}
may be aided by the norm constraint $\|U\|^{2}=r^{\star}$ imposed
alongside nonnegativity $U\ge0$, which help prevent the type of collapse
exhibited here.

\subsection{Organization}

The remainder of this paper is organized is follows. In \secref{prelim},
we formally define what it means for a point $U$ to be located at
a spurious local minimum, and to state the necessary and the second-order
sufficient conditions for local optimality. In \secref{counter},
we prove \thmref{main} by verifying the sufficient conditions at
the spurious point. In \secref{counter}, we prove \propref{spu2}
directly. In \secref{nospu}, we prove \propref{nospu} by verifying
that the necessary conditions imply global optimality when $\delta=0$
and $r^{\star}=1$. Finally, in \secref{conclu}, we conclude by speculating
on possible future directions that might lead to benign nonconvexity
guarantees. 

\section{Preliminaries}\label{sec:prelim}

\subsection{Necessary and sufficient conditions for local optimality}

Consider minimizing a twice differentiable cost $f:\R_{+}^{n\times r}\to\R$.
The point $U_{0}\in\R_{+}^{n\times r}$ is said to be at a \emph{local
minimum} (resp. is a \emph{local minimizer}) if its value is minimum
within a feasible neighborhood 
\[
f(U_{0})\le f(U)\quad\text{for all }U\in\R_{+}^{n\times r}\cap B(U_{0},\epsilon)
\]
where $B(U_{0},\epsilon)$ denotes the open Euclidean ball of radius
$\epsilon>0$ centered at $U_{0}$. It is further said to be a \emph{strict}
local minimum if the inequality is strict for all $U\ne U_{0}$. In
either case, we say that $U_{0}$ is \emph{spurious} if it is strictly
worse than the global minimum $f(U_{0})>f(U_{\star})$. 

Our proof verifies local optimality through the \emph{second-order
sufficient condition}. The first-order part of the condition reads
\begin{equation}
U\in\R_{+}^{n\times r},\quad\nabla f(U)\in\R_{+}^{n\times r},\quad\inner{\nabla f(U)}U=0,\label{eq:foc}
\end{equation}
while the second-order part additionally requires
\begin{gather}
\inner{\nabla^{2}f(U)[\dot{U}]}{\dot{U}}>0\text{ for all }\dot{U}\ne0,\;\dot{U}\in\Cone(U,\nabla f(U)),\label{eq:socs}\\
\text{where }\Cone(U,G)=\left\{ \dot{U}\in\R^{n\times r}:\begin{array}{cc}
\dot{U}_{i,j}=0 & \text{ for }G_{i,j}>0,\\
\dot{U}_{i,j}\ge0 & \text{ for }U_{i,j}=0.
\end{array}\right\} \nonumber 
\end{gather}
The set $\Cone(U,G)$ here is called the \emph{critical cone}. It
follows from standard Lagrangian theory that if $U_{0}$ satisfies
(\ref{eq:foc}) and (\ref{eq:socs}) then it is a strict local minimizer,
see~\cite[Theorem 12.6]{nocedal2006numerical}. In fact, as clarified
by the copositivity-based framework developed in~\cite[Theorem~1.1]{bomze2016copositivity},
these conditions imply the stronger conclusion that $U_{0}$ is a
\emph{strong} local minimizer, in the sense that a local quadratic
growth bound $f(U)\ge f(U_{0})+\frac{\mu}{2}\|U-U_{0}\|^{2}$ holds
in a neighborhood of $U_{0}$. However, as pointed out by Murty and
Kabadi~\cite{murty1987some}, verifying that (\ref{eq:socs}) holds
is NP-hard in general, as it amounts to testing whether a submatrix
of the Hessian is strictly copositive. 

In reverse, every local minimizer $U_{0}$ must satisfy the \emph{second-order
necessary condition}, which shares the same first-order condition
(\ref{eq:foc}) but has the following nonstrict second-order part
\begin{gather}
\inner{\nabla^{2}f(U)[\dot{U}]}{\dot{U}}\ge0\text{ for all }\dot{U}\in\Cone(U,\nabla f(U)).\label{eq:socn}
\end{gather}
A point $U_{0}$ that satisfies (\ref{eq:foc}) and (\ref{eq:socn})
is called a \emph{second-order critical point}; it can be either a
local minimizer or a higher-order saddle point, see~\cite[Theorem 12.5]{nocedal2006numerical}.
The fact that all local minimizers satisfy (\ref{eq:foc}) and (\ref{eq:socn})
is due to the linearity of the constraint $U\ge0$, so the tangent
and linearization cones coincide, and the Abadie constraint qualification
holds. Our proof of the inexistence of spurious local minima for $\delta=0$
and $r^{\star}=1$ work by using (\ref{eq:foc}) and (\ref{eq:socn})
to imply $UU^{T}=U_{\star}U_{\star}^{T}$.

\subsection{Gradients and Hessians}

To verify that a point $U$ satisfies the second-order sufficient
condition (\ref{eq:foc}) and (\ref{eq:socs}), we will need to evaluate
the gradient and the Hessian quadratic form of $f(U)=\frac{1}{2}\|\A(UU^{T}-U_{\star}U_{\star}^{T})\|^{2}$,
written
\begin{gather}
\nabla f(U)=(S_{U}+S_{U}^{T})U,\qquad S_{U}=\AT\A(UU^{T}-U_{\star}U_{\star}^{T}),\label{eq:fderiv}\\
\inner{\nabla^{2}f(U)[\dot{U}]}{\dot{U}}=\inner{S_{U}+S_{U}^{T}}{\dot{U}\dot{U}^{T}}+\|\A(U\dot{U}^{T}+\dot{U}U^{T})\|^{2}.\nonumber 
\end{gather}
Here, $\AT:\R^{m}\to\R^{n\times n}$ denotes the adjoint map of $\A$.
The gradient and Hessian quadratic form of the function $g(U)=\frac{1}{2}\|\A(LR^{T}-L_{\star}R_{\star}^{T})\|^{2}$
at $U=[L;R]$ and $\dot{U}=[\dot{L};\dot{R}]$ are respectively
\begin{gather}
\nabla g(U)=\begin{bmatrix}S_{L,R}R\\
S_{L,R}^{T}L
\end{bmatrix},\qquad S_{L,R}=\AT\A(LR^{T}-L_{\star}R_{\star}^{T}),\label{eq:gderiv}\\
\inner{\nabla^{2}g(U)[\dot{U}]}{\dot{U}}=2\inner{S_{L,R}}{\dot{L}\dot{R}^{T}}+\|\A(L\dot{R}^{T}+\dot{L}R^{T})\|^{2}.\nonumber 
\end{gather}
We will also need to evaluate the Hessian cross terms, for $\dot{U}_{1}=[\dot{L}_{1};\dot{R}_{1}]$
and $\dot{U}_{2}=[\dot{L}_{2};\dot{R}_{2}]$:
\begin{equation}
\inner{\nabla^{2}g(U)[\dot{U}_{1}]}{\dot{U}_{2}}=\inner{S_{L,R}}{\dot{L}_{1}\dot{R}_{2}^{T}+\dot{L}_{2}\dot{R}_{1}^{T}}+\inner{\A(L\dot{R}_{1}^{T}+\dot{L}_{1}R^{T})}{\A(L\dot{R}_{2}^{T}+\dot{L}_{2}R^{T})}.\label{eq:gcross}
\end{equation}
For the balancing regularizer $h([L;R])=\frac{1}{2}\|\A(LR^{T}-L_{\star}R_{\star}^{T})\|^{2}$
at $U=[L;R]$ and $\dot{U}=[\dot{L};\dot{R}]$, these are respectively
\begin{gather}
\nabla h(U)=\begin{bmatrix}LD\\
-RD
\end{bmatrix},\qquad D=L^{T}L-R^{T}R,\label{eq:hderiv}\\
\inner{\nabla^{2}h(U)[\dot{U}]}{\dot{U}}=2\inner D{\dot{L}^{T}\dot{L}-\dot{R}^{T}\dot{R}}+4\|L^{T}\dot{L}-R^{T}\dot{R}\|^{2}.\nonumber 
\end{gather}

\section{Counterexamples for $\delta>0$ and any $r\ge r^{\star}\ge1$}\label{sec:counter}

\subsection{Symmetric factorization}

\global\long\def\H{\mathcal{H}}%
Without loss of generality, we can set $Q=[Q_{1},Q_{2}]=[\e_{1},\dots,\e_{r+r^{\star}}]$
where $\e_{i}$ is the $i$-th column of $I_{n}$. Our proof of \thmref{main}
relies on the insight that the geometric landscape of $f(U)$ is entirely
determined by the kernel operator $\H=\AT\A$, rather than the specific
choice of $\A$ itself. In other words, while there are many linear
maps $\A$ that satisfy $\AT\A=\H$, the objective value, gradient,
and Hessian of $f$ depend only on $\H$; the specific realization
$\A$ and its adjoint $\A^{*}$ are never ``seen.''

The following is the critical structural property on $\H$ that underpins
this paper.

\begin{assume}\label{asm:Anrm}For $\alpha,\epsilon>0$, let the
kernel operator $\H\equiv\AT\A,$ $\H:\R^{n\times n}\to\R^{n\times n}$
satisfy 
\begin{gather*}
\H(X)=X+C\inner B{Q^{T}XQ}+B\inner C{Q^{T}XQ}\quad\text{for all }X\in\R^{n\times n}\\
\text{where }B=\begin{bmatrix}0 & -\epsilon\one_{r^{\star}}\one_{r}^{T}\\
-\epsilon\one_{r}\one_{r^{\star}}^{T} & I_{r}
\end{bmatrix},\qquad C=\frac{\alpha^{2}}{r^{\star}}\begin{bmatrix}I_{r^{\star}} & 0\\
0 & 0
\end{bmatrix}.
\end{gather*}
\end{assume}

We first verify that that this structural property also implies the
restricted isometry property. 
\begin{lem}
\label{lem:delta}Suppose that \asmref{Anrm} holds. Then, (\ref{eq:rip})
holds with constant $\delta=\|B\|\|C\|=\alpha^{2}\sqrt{r/r^{\star}+2\epsilon^{2}r}$.
Additionally, if $\delta<1$, then $\H$ is positive definite.
\end{lem}

\begin{proof}
\asmref{Anrm} can be rewritten as $\vec[\H(X)]=[I+bc^{T}+cb^{T}]\vec(X)$
where $b=\vec(QBQ^{T})$ and $c=\vec(QCQ^{T})$. Therefore, $\H$
has eigenvalues
\[
\mathrm{eig}(\H)=\{1+b^{T}c\pm\|b\|\|c\|,1,\dots,1\}.
\]
In our case, $b^{T}c=0$ and $\|b\|=\sqrt{r+2\epsilon^{2}rr^{\star}}$
and $\|c\|=\alpha^{2}/\sqrt{r^{\star}}$, so (\ref{eq:rip}) holds
with $\delta=\|b\|\|c\|$. Moveover, if $\delta<1$, then $\mathrm{eig}(\H)>0$. 
\end{proof}
Now, we show that $U_{0}=\alpha Q_{2}$ is a strict local minimizer
for any $\A$ whose $\H=\AT\A$ satisfies \asmref{Anrm}. Our special
choice of $B$ and $C$ in \asmref{Anrm} critically drives the delicate
cancellations needed for the second-order sufficient condition (\ref{eq:foc})-(\ref{eq:socs})
to hold.
\begin{lem}
\label{lem:sym}Suppose that \asmref{Anrm} holds. Then, the function
$f(U)=\frac{1}{2}\|\A(UU^{T}-U_{\star}U_{\star})\|^{2}$ with $U_{\star}=Q_{1}$
satisfies at the point $U_{0}=\alpha Q_{2}$:
\begin{gather*}
f(U_{0})\ge\frac{r^{\star}}{2}(1-\delta^{2}),\quad\nabla f(U_{0})\ge0,\quad\inner{\nabla f(U_{0})}{U_{0}}=0,\\
\inner{\nabla^{2}f(U_{0})[\dot{U}]}{\dot{U}}\ge\alpha^{2}\|\dot{U}\|^{2}\text{ for all }\dot{U}\in\Cone(U_{0},\nabla f(U_{0}))
\end{gather*}
where $\delta=\alpha^{2}\sqrt{r/r^{\star}+2\epsilon^{2}r}$.
\end{lem}

\begin{proof}
We verify for 
\[
E=U_{0}U_{0}^{T}-U_{\star}U_{\star}^{T}=\begin{bmatrix}-I_{r^{\star}} & 0 & 0\\
0 & \alpha^{2}I_{r} & 0\\
0 & 0 & 0
\end{bmatrix},
\]
that we have
\begin{align*}
S_{U}=\AT\A(E) & =\begin{bmatrix}-I_{r^{\star}} & 0 & 0\\
0 & \alpha^{2}I_{r} & 0\\
0 & 0 & 0
\end{bmatrix}+\alpha^{2}r\begin{bmatrix}\frac{\alpha^{2}}{r^{\star}}I_{r^{\star}} & 0 & 0\\
0 & 0 & 0\\
0 & 0 & 0
\end{bmatrix}-\alpha^{2}\begin{bmatrix}0 & -\epsilon\one\one^{T} & 0\\
-\epsilon\one\one^{T} & I_{r} & 0\\
0 & 0 & 0
\end{bmatrix}\\
 & =\alpha^{2}\begin{bmatrix}-\beta I_{r^{\star}} & \epsilon\one\one^{T} & 0\\
\epsilon\one\one^{T} & 0 & 0\\
0 & 0 & 0
\end{bmatrix}\qquad\text{ where }\beta=\alpha^{-2}-\alpha^{2}r/r^{\star}.
\end{align*}
Then, the function value at $U_{0}$ is
\[
f(U_{0})=\frac{1}{2}\|\A(E)\|^{2}=\frac{1}{2}\inner{S_{U}}E=\frac{1}{2}(r^{\star}-\alpha^{4}r)\ge\frac{r^{\star}}{2}(1-\delta^{2}).
\]
the last inequality substitutes $\delta=\alpha^{2}\sqrt{r/r^{\star}+2\epsilon^{2}r}\ge\alpha^{2}\sqrt{r/r^{\star}}.$
To verify (\ref{eq:foc}), we use (\ref{eq:fderiv}) to compute the
gradient at $U_{0}$ to be
\[
\nabla f(U_{0})=2\alpha^{2}\begin{bmatrix}\beta I & \epsilon\one\one^{T} & 0\\
\epsilon\one\one^{T} & 0 & 0\\
0 & 0 & 0
\end{bmatrix}\begin{bmatrix}0\\
\alpha I\\
0
\end{bmatrix}=2\alpha^{3}\begin{bmatrix}\epsilon\one\one^{T}\\
0\\
0
\end{bmatrix},
\]
and therefore $\nabla f(U_{0})\ge0$ and $\inner{\nabla f(U_{0})}{U_{0}}=0$.
To verify (\ref{eq:socs}), we first compute the associated critical
cone to be
\[
\Cone(U_{0},\nabla f(U_{0}))=\left\{ \begin{bmatrix}0\\
Y\\
Z
\end{bmatrix}:\begin{array}{c}
Y\in\R^{r\times r},\;Y_{i,j}\ge0\text{ for all }i\ne j\\
Z\in\R_{+}^{(n-r-r^{\star})\times r}
\end{array}\right\} .
\]
Therefore, for any $\dot{U}\in\Cone(U_{0},\nabla f(U_{0}))$, we have
\begin{gather*}
\inner{S_{U}}{\dot{U}\dot{U}^{T}}=\inner{\alpha^{2}\begin{bmatrix}\beta I & \epsilon\one\one^{T} & 0\\
\epsilon\one\one^{T} & 0 & 0\\
0 & 0 & 0
\end{bmatrix}}{\begin{bmatrix}0\\
Y\\
Z
\end{bmatrix}\begin{bmatrix}0\\
Y\\
Z
\end{bmatrix}^{T}}=0,\\
\|\A(U_{0}\dot{U}^{T}+\dot{U}U_{0}^{T})\|^{2}=\left\Vert \A\left(\begin{bmatrix}0 & 0 & 0\\
0 & \alpha(Y+Y^{T}) & \alpha Z^{T}\\
0 & \alpha Z & 0
\end{bmatrix}\right)\right\Vert ^{2}=\alpha^{2}\left\Vert \begin{bmatrix}Y+Y^{T} & Z^{T}\\
Z & 0
\end{bmatrix}\right\Vert ^{2}.
\end{gather*}
Substituting the above into the Hessian quadratic form in (\ref{eq:fderiv})
yields
\begin{align*}
\inner{\nabla^{2}f(U_{0})[\dot{U}]}{\dot{U}} & =2\inner{S_{U}}{\dot{U}\dot{U}^{T}}+\|\A(U_{0}\dot{U}^{T}+\dot{U}U_{0}^{T})\|^{2}\\
 & =0+2\alpha^{2}\left(\|Y\|^{2}+\inner Y{Y^{T}}+\|Z\|^{2}\right)\\
 & \ge2\alpha^{2}(\|Y\|^{2}+\|Z\|^{2})=2\alpha^{2}\|\dot{U}\|^{2}.
\end{align*}
The second line substituted $\|Y+Y^{T}\|^{2}=\|Y\|^{2}+\|Y^{T}\|^{2}+2\inner Y{Y^{T}}$.
The final line recalls that $Y_{i,j}\ge0$ for all $i\ne j$, and
therefore 
\[
\inner Y{Y^{T}}=\sum_{i,j}Y_{i,j}Y_{j,i}=\sum_{i}Y_{i,i}^{2}+\sum_{i\ne j}Y_{i,j}Y_{j,i}\ge0.
\]
\end{proof}
To summarize, suppose that $\A:\R^{n\times n}\to\R^{m}$ satisfies
\asmref{Anrm} with $\epsilon>0$ and $0<\alpha<(r/r^{\star}+2\epsilon^{2}r)^{-1/4}$.
By \lemref{delta}, $\A$ also satisfies (\ref{eq:rip}) with $\delta=\alpha^{2}\sqrt{r/r^{\star}+2\epsilon^{2}r}<1$.
Nevertheless, by \lemref{sym}, the second-order sufficient condition
(\ref{eq:foc})-(\ref{eq:socs}) holds at a point $U_{0}$ with $f(U_{0})>f(U_{\star})=0$,
so $U_{0}$ is indeed a spurious strict local minimizer. Choosing
$\alpha\to0^{+}$ also sets the RIP constant $\delta\to0^{+}$, so
we conclude that there is no $\delta>0$ that would eliminate all
spurious local minima. 

Formally, a valid choice of $\A$ as describe above is already guaranteed
to exist, because $\H$ is positive definite by \lemref{delta}, so
a decomposition $\H=\AT\A$ always exists. Nevertheless, to provide
a transparent, computationally reproducible example (as used to generate
\figref{counter}), \thmref{main} provides a specific realization
of $\A$. The following lemma confirms this explicit choice is valid. 
\begin{lem}
\label{lem:check}The $\A$ defined in \thmref{main} satisfies \asmref{Anrm}.
\end{lem}

\begin{proof}
\asmref{Anrm} can be rewritten as $\vec[\H(X)]=[I+bc^{T}+cb^{T}]\vec(X)$
where $b=\vec(QBQ^{T})$ and $c=\vec(QCQ^{T})$. The $\A$ defined
in \thmref{main} satisfies in turn
\[
\vec(\A(X))=[I+(b+\gamma c)c^{T}]\vec(X)\qquad\text{ where }\gamma=\frac{-1+\sqrt{1-\|b\|^{2}\|c\|^{2}}}{\|c\|^{2}}.
\]
Here, we recall from \lemref{delta} that $\delta=\|b\|\|c\|$. Finally,
we take the difference
\begin{align*}
 & (I+bc^{T}+\gamma cc^{T})^{T}(I+bc^{T}+\gamma cc^{T})-(I+bc^{T}+cb^{T})\\
= & 2\gamma cc^{T}+(bc^{T}+\gamma cc^{T})^{T}(bc^{T}+\gamma cc^{T})\\
= & (\|c\|^{2}\gamma^{2}+2\gamma+\|b\|^{2})cc^{T}=0
\end{align*}
The final line uses $b^{T}c=0$, and then substituted our stated $\gamma$
as a solution to the quadratic equation. 
\end{proof}

\subsection{Asymmetric factorization}

Our proof of the asymmetric part of \thmref{main} is based on the
insight that, for a fixed linear operator $\A$ and target matrix
$M_{\star}=U_{\star}U_{\star}^{T}$, it is possible for the asymmetric
factorization $LR^{T}$ to \emph{inherit} a local minimum from the
symmetric factorization $UU^{T}$. In other words, if $U=U_{0}$ is
a local minimizer for $UU^{T}$, then $L=R=U_{0}$ is a local minimizer
for $LR^{T}$. In order for this inheritance to be possible, the operator
$\A$ must satisfy the following key property. 

\begin{assume}\label{asm:sym}Let $\A:\R^{n\times n}\to\R^{m}$ satisfy
$\AT\A(E)\in\S^{n}$ for all $E\in\S^{n}$.\end{assume}

In fact, this key property is already implied by our critical assumption
from the previous section. Nevertheless, we prefer to isolate \asmref{sym},
because the technique of passing local minima down from $UU^{T}$
to $LR^{T}$ is broadly useful in the study of benign nonconvexity
in low-rank factorizations.
\begin{lem}
\asmref{Anrm} implies \asmref{sym}.
\end{lem}

\begin{proof}
\asmref{Anrm} specifies $\|\A(S+K)\|^{2}=\|\A(S)\|^{2}+\|\A(K)\|^{2}$
for all symmetric $S=S^{T}$ and skew symmetric $K=-K^{T}$. This
implies that $\inner{\A(S)}{\A(K)}=0$ for all $S,K$, which holds
only if $\AT\A(S)=[\AT\A(S)]^{T}$ and $\AT\A(K)=-[\AT\A(K)]^{T}$. 
\end{proof}
Under \asmref{sym}, there is a simple connection between the derivatives
of the symmetric $UU^{T}$ and the asymmetric $LR^{T}$ cases. Note
that we do not yet impose nonnegativity on the factor matrices $U,L,R$. 
\begin{lem}
\label{lem:sym2asym}Suppose that \asmref{sym} holds. Define the
two functions 
\[
f(U)=\frac{1}{2}\|\A(UU^{T}-M^{\star})\|^{2},\quad g([L;R])=\frac{1}{2}\|\A(LR^{T}-M^{\star})\|^{2},
\]
with the same $M^{\star}$. Then, their gradients are related at $U^{(2)}=[U;U]$
as
\[
\nabla g(U^{(2)})=\frac{1}{2}[\nabla f(U);\nabla f(U)],
\]
and their Hessian quadratic forms are related at $U^{(2)}=[U;U]$
for all $\dot{L},\dot{R}$ as 
\begin{align*}
\inner{\nabla^{2}g(U^{(2)})[\dot{L};\dot{R}]}{[\dot{L};\dot{R}]} & =\inner{\nabla^{2}f(U)[\dot{U}]}{\dot{U}}-2\inner{\A(UU^{T}-M^{\star})}{\A(\dot{V}\dot{V}^{T})}\\
 & \qquad\qquad+\|\A(\dot{V}U^{T}-U\dot{V}^{T})\|^{2}
\end{align*}
where $\dot{U}=\frac{1}{2}(\dot{L}+\dot{R})$ and $\dot{V}=\frac{1}{2}(\dot{L}-\dot{R})$.
\end{lem}

\begin{proof}
Under \asmref{sym}, the matrix $S=\AT\A(UU^{T}-M^{\star})$ is symmetric,
and therefore
\[
\nabla g(U^{(2)})=\begin{bmatrix}SU\\
S^{T}U
\end{bmatrix}=\frac{1}{2}\begin{bmatrix}(S+S^{T})U\\
(S+S^{T})U
\end{bmatrix}=\frac{1}{2}\begin{bmatrix}\nabla f(U)\\
\nabla f(U)
\end{bmatrix}.
\]
Next, using (\ref{eq:gcross}), we find that the following Hessian
cross term is zero, 
\begin{align*}
\inner{\nabla^{2}g(U^{(2)})[\begin{bmatrix}\dot{U}\\
\dot{U}
\end{bmatrix}]}{\begin{bmatrix}\dot{V}\\
-\dot{V}
\end{bmatrix}} & =\langle S,\ddot{K}\rangle+\inner{\A(\dot{U}U^{T}+U\dot{U}^{T})}{\A(\dot{K})}\\
 & =0+\inner{\AT\A(\dot{U}U^{T}+U\dot{U}^{T})}{\dot{K}}=0,
\end{align*}
because $\ddot{K}=\dot{V}\dot{U}^{T}-\dot{U}\dot{V}^{T}$ and $\dot{K}=\dot{V}U^{T}-U\dot{V}^{T}$
are skew-symmetric, while both $S$ and $\AT\A(\dot{U}U^{T}+U\dot{U}^{T})$
are symmetric under \asmref{sym}. Therefore, the Hessian quadratic
form comprises just the two block-diagonal terms, which we can again
evaluate using (\ref{eq:gcross}):
\begin{gather*}
\inner{\nabla^{2}g(U^{(2)})[\begin{bmatrix}\dot{L}\\
\dot{R}
\end{bmatrix}]}{\begin{bmatrix}\dot{L}\\
\dot{R}
\end{bmatrix}}=\inner{\nabla^{2}g(U^{(2)})[\begin{bmatrix}\dot{U}\\
\dot{U}
\end{bmatrix}]}{\begin{bmatrix}\dot{U}\\
\dot{U}
\end{bmatrix}}+\inner{\nabla^{2}g(U^{(2)})[\begin{bmatrix}\dot{V}\\
-\dot{V}
\end{bmatrix}]}{\begin{bmatrix}\dot{V}\\
-\dot{V}
\end{bmatrix}}\\
=\inner{\nabla^{2}f(U)[\dot{U}]}{\dot{U}}+2\inner S{-\dot{V}\dot{V}^{T}}+\|\A(\dot{V}U^{T}-U\dot{V}^{T})\|^{2}.
\end{gather*}
\end{proof}
We also recall the derivatives of the balancing regularizer from (\ref{eq:hderiv}). 
\begin{lem}
\label{lem:bal}Let $h([L;R])=\frac{1}{2}\|L^{T}L-R^{T}R\|^{2}$.
Its gradient and Hessian quadratic form are evaluated at $U^{(2)}=[U;U]$
as follows for all $\dot{L},\dot{R}$:
\[
\nabla h(U^{(2)})=0,\quad\inner{\nabla^{2}h(U^{(2)})[\dot{L};\dot{R}]}{[\dot{L};\dot{R}]}=4\|U^{T}(\dot{L}-\dot{R})\|^{2}.
\]
\end{lem}

We are now ready to verify that the asymmetric case $LR^{T}$ inherits
the same spurious local minimizer $L=R=U_{0}=\alpha Q_{2}$ from the
symmetric case by verifying the first-order necessary condition (\ref{eq:foc})
and the second-order sufficient condition (\ref{eq:socs}) .
\begin{lem}
\label{lem:asym}Suppose that \asmref{Anrm} holds. Then, the function
\[
g_{\lambda}:\R^{n\times r}\to\R,\qquad[L;R]\mapsto\frac{1}{2}\|\A(LR^{T}-Q_{1}Q_{1}^{T})\|^{2}+\frac{\lambda}{2}\|L^{T}L-R^{T}R\|^{2}
\]
with $\lambda\ge0$ satisfies at the point $U_{0}^{(2)}=\alpha[Q_{2};Q_{2}]$
\begin{gather*}
g_{\lambda}(U_{0}^{(2)})\ge\frac{r^{\star}}{2}(1-\delta^{2}),\quad\nabla g_{\lambda}(U_{0}^{(2)})\ge0,\quad\inner{\nabla g_{\lambda}(U_{0}^{(2)})}{U_{0}^{(2)}}=0,\\
\inner{\nabla^{2}g_{\lambda}(U_{0}^{(2)})[\dot{U}^{(2)}]}{\dot{U}^{(2)}}\ge\mu\|\dot{U}^{(2)}\|^{2}\text{ for all }\dot{U}^{(2)}\in\Cone(U_{0}^{(2)},\nabla g_{\lambda}(U_{0}^{(2)}))
\end{gather*}
where $\delta=\alpha^{2}\sqrt{r/r^{\star}+2\epsilon^{2}r}$ and $\mu=\min\{\lambda,\frac{1}{4}\}\cdot\alpha^{2}$.
\end{lem}

\begin{proof}
\sloppy Let $f(U)=\frac{1}{2}\|\A(UU^{T}-Q_{1}Q_{1}^{T})\|^{2}$.
It immediately follows from \lemref{sym}, \lemref{sym2asym}, and
\lemref{bal} that $g_{\lambda}$ satisfies at $U_{0}^{(2)}=[U_{0};U_{0}]$:
\begin{gather*}
g_{\lambda}(U_{0}^{(2)})\ge g_{0}(U_{0}^{(2)})=f_{0}(U_{0})\ge\frac{r^{\star}}{2}(1-\delta^{2}),\\
\nabla g_{\lambda}(U_{0}^{(2)})=\frac{1}{2}\begin{bmatrix}\nabla f(U_{0})\\
\nabla f(U_{0})
\end{bmatrix}\ge0,\quad\inner{\nabla g_{\lambda}(U_{0}^{(2)})}{U_{0}^{(2)}}=\frac{1}{2}\inner{\begin{bmatrix}\nabla f(U_{0})\\
\nabla f(U_{0})
\end{bmatrix}}{\begin{bmatrix}U_{0}\\
U_{0}
\end{bmatrix}}=0.
\end{gather*}
Moreover, this shows that the critical cone for $g_{\lambda}$ is
just the Cartesian product of two critical cones for $f$:
\[
\Cone(U_{0}^{(2)},\nabla g_{\lambda}(U_{0}^{(2)}))=\{[\dot{L};\dot{R}]:\dot{L},\dot{R}\in\Cone(U_{0},\nabla f(U_{0}))\}.
\]
For $\dot{L},\dot{R}\in\Cone(U_{0},\nabla f(U_{0}))$, note that $\dot{U}=\frac{1}{2}(\dot{L}+\dot{R})\in\Cone(U_{0},\nabla f(U_{0}))$
because this is a cone. However, the same is not necessarily true
of $\dot{V}=\frac{1}{2}(\dot{L}-\dot{R})$. Nevertheless, we recall
from the proof of \lemref{sym} that the sparsity patterns of $S=\AT\A(UU^{T}-M^{\star})$
and $\dot{L},\dot{R}$ are disjoint:
\[
S=\alpha^{2}\begin{bmatrix}\beta I & \epsilon11^{T} & 0\\
\epsilon11^{T} & 0 & 0\\
0 & 0 & 0
\end{bmatrix},\quad\dot{L}=\begin{bmatrix}0\\
Y_{L}\\
Z_{L}
\end{bmatrix},\quad\dot{R}=\begin{bmatrix}0\\
Y_{R}\\
Z_{R}
\end{bmatrix},
\]
so that $\inner S{\dot{V}\dot{V}^{T}}=0$. Moreover, writing $Y=Y_{L}-Y_{R}$
and $Z=Z_{L}-Z_{R}$, we have
\[
\|\A(\dot{V}U_{0}^{T}-U_{0}\dot{V}^{T})\|^{2}=\left\Vert \A\left(\begin{bmatrix}0 & 0 & 0\\
0 & \alpha(Y-Y^{T}) & -\alpha Z^{T}\\
0 & \alpha Z & 0
\end{bmatrix}\right)\right\Vert ^{2}\ge2\alpha^{2}\|Z\|^{2}.
\]
The balancing regularizer contributes the following term
\[
\inner{\nabla^{2}h(U_{0}^{(2)})[\dot{L};\dot{R}]}{[\dot{L};\dot{R}]}=4\|U_{0}^{T}(\dot{L}-\dot{R})\|^{2}=8\alpha^{2}\|Y\|^{2}.
\]
Finally, piecing everything together,
\begin{gather*}
\inner{\nabla^{2}g_{\lambda}(U^{(2)})[\dot{L};\dot{R}]}{[\dot{L};\dot{R}]}=\underbrace{\inner{\nabla^{2}f(U_{0})[\dot{U}]}{\dot{U}}}_{\ge2\alpha^{2}\|\dot{U}\|^{2}}\\
-2\underbrace{\inner S{\dot{V}\dot{V}^{T}}}_{=0}+\underbrace{\|\A(\dot{V}U_{0}^{T}-U_{0}\dot{V}^{T})\|^{2}}_{\ge2\alpha^{2}\|Z\|^{2}}+\lambda\underbrace{\inner{\nabla^{2}h(U_{0}^{(2)})[\dot{L};\dot{R}]}{[\dot{L};\dot{R}]}}_{=8\alpha^{2}\|Y\|^{2}}.
\end{gather*}
Substituting $\frac{1}{4}\|\dot{V}\|^{2}=\|Y\|^{2}+\|Z\|^{2}$ and
$\frac{1}{2}\|[\dot{L};\dot{R}]\|^{2}=\|\dot{U}\|^{2}+\|\dot{V}\|^{2}$
completes the proof. 
\end{proof}
To summarize, given $\epsilon>0$ and $0<\alpha<(r/r^{\star}+2\epsilon^{2}r)^{-1/4}$,
\lemref{check} verifies that the $\A:\R^{n\times n}\to\R^{n\times n}$
stated in \thmref{main} satisfies \asmref{Anrm}, and \lemref{delta}
ensures that $\A$ also satisfies (\ref{eq:rip}) with $\delta=\alpha^{2}\sqrt{r/r^{\star}+2\epsilon^{2}r}<1$.
Nevertheless, by \lemref{asym}, the second-order necessary condition
(\ref{eq:foc}), (\ref{eq:socn}) holds at a point $(L,R)=(U_{0},U_{0})$
with $g_{\lambda}(U_{0},U_{0})>g_{\lambda}(U_{\star},U_{\star})=0$,
so this is indeed a spurious second-order critical point. If additionally
$\lambda>0$, then the second-order sufficient condition (\ref{eq:foc})-(\ref{eq:socs})
holds, so $(L,R)=(U_{0},U_{0})$ is also a spurious strict local minimizer.
In either case, choosing $\alpha\to0^{+}$ also sets $\delta\to0^{+}$,
so we conclude that there is no $\delta>0$ that would eliminate all
spurious second-order critical points (resp. all spurious local minima). 

\section{Counterexamples for $\delta=0$ and any $r\ge r^{\star}>1$}\label{sec:counter2}

We begin with a basic counterexample in the fully observed symmetric
setting, which isolates a geometric failure caused solely by the nonnegativity
constraint. To ensure that the local minimum identified below is genuinely
spurious, we further require $r\ge r^{\star}=k+1$, or else the global
minimum of zero may not be attainable at $U=U_{\star}$. This construction
also explains why the rank-$1$ case is special; the point $U_{0}$
stated below with $k=0$ is actually globally optimal, and this is
consistent with benign behavior for $r^{\star}=1$.
\begin{lem}
Pick $m>r\ge1$ and $k\ge1$. The function 
\[
f:\R_{+}^{(m+k)\times r}\to\R,\;U\mapsto\frac{1}{2}\|UU^{T}-U_{\star}U_{\star}^{T}\|^{2}
\]
 admits a local minimum at $U_{0}$, where
\begin{equation}
U_{0}=\begin{bmatrix}\frac{1}{\sqrt{r}}\one_{m}\one_{r}^{T}\\
0_{k\times r}
\end{bmatrix},\quad U_{\star}=\begin{bmatrix}\one_{m} & 0\\
0 & I_{k}
\end{bmatrix}.\label{eq:Udefs}
\end{equation}
\end{lem}

\begin{proof}
Partition each $U=\begin{bmatrix}Y\\
Z^{T}
\end{bmatrix}$ for $Y\in\R_{+}^{m\times r}$ and $Z\in\R_{+}^{n\times k}$, and
observe that 
\begin{align*}
f(U)-f(U_{0})= & \frac{1}{2}\|UU^{T}-U_{\star}U_{\star}^{T}\|^{2}-\frac{1}{2}\|U_{0}U_{0}^{T}-U_{\star}U_{\star}^{T}\|^{2}\\
= & \frac{1}{2}\left(\|YY^{T}-\one\one^{T}\|^{2}+2\|YZ\|^{2}+\|Z^{T}Z-I_{k}\|^{2}\right)-\frac{1}{2}\|I_{k}\|^{2}\\
= & \frac{1}{2}\left(\|YY^{T}-\one\one^{T}\|^{2}+\|Z^{T}Z-I_{k}\|^{2}\right)+\|YZ\|^{2}-\|Z\|^{2}\\
\ge & \|YZ\|^{2}-\|Z\|^{2}.
\end{align*}
At $U=U_{0}$, we have $Y=Y_{0}=\frac{1}{\sqrt{r}}\one_{m}\one_{r}^{T}$,
and hence
\begin{align}
\|YZ\|^{2}=\|\frac{1}{\sqrt{r}}\one_{m}\one_{r}^{T}Z\|^{2}=\frac{m}{r}\|\one_{r}^{T}Z\|^{2} & =\frac{m}{r}\sum_{i=1}^{k}\|Ze_{i}\|_{\ell_{1}}^{2}\text{ since }Z_{i,j}\ge0\nonumber \\
 & \ge\frac{m}{r}\sum_{i=1}^{k}\|Ze_{i}\|_{\ell_{2}}^{2}=\frac{m}{r}\|Z\|^{2}.\label{eq:YZbnd}
\end{align}
Within the neighborhood $\|U-U_{0}\|\le\rho=\frac{1}{4r\sqrt{m}}$,
we have
\begin{align}
\|Y^{T}Y-Y_{0}^{T}Y_{0}\| & =\|\mathrm{sym}[(Y-Y_{0})(Y+Y_{0})^{T}]\|\le\|Y-Y_{0}\|\|Y+Y_{0}\|\nonumber \\
 & \le\rho(2\|Y_{0}\|+\rho)=\rho(2\sqrt{m}+\rho)\le\frac{1}{r}\quad\text{ for }m,r\ge1.\label{eq:radiusbnd}
\end{align}
Combined, the local suboptimality is lower-bounded
\begin{align*}
f(U)-f(U_{0}) & \ge\|YZ\|^{2}-\|Z\|^{2}=\inner{Y^{T}Y-I_{r}}{ZZ^{T}}\\
 & \ge\|Y_{0}Z\|^{2}-\|Z\|^{2}-\|Y^{T}Y-Y_{0}^{T}Y_{0}\|\|Z\|^{2}\\
 & \ge\frac{1}{r}(m-r-1)\|Z\|^{2}\ge0\quad\text{ for integer }m>r\ge1,
\end{align*}
so we conclude that $U_{0}$ is a local minimum within this neighborhood.
\end{proof}
The argument extends directly to the asymmetric factorization. To
keep the notation light, we drop the factor $\frac{1}{2}$ from the
definition of $g$, since it plays no role in the local analysis. 
\begin{lem}
Pick $m>r\ge1$ and $k\ge1$. The function 
\[
g:\R_{+}^{(m+k)\times r}\times\R_{+}^{(m+k)\times r}\to\R,\;(L,R)\mapsto\|LR^{T}-U_{\star}U_{\star}^{T}\|^{2}
\]
admits a local minimum at $L=R=U_{0}$, where $U_{0},U_{\star}$ are
the same as (\ref{eq:Udefs}).
\end{lem}

\begin{proof}
Using the same partition of $U$ for both $L$ and $R$, we have $g(U_{0},U_{0})=\|I_{k}\|^{2}$
and
\begin{align*}
g(L,R)-g(U_{0},U_{0})= & \|LR^{T}-U_{\star}U_{\star}^{T}\|^{2}-\|U_{0}U_{0}^{T}-U_{\star}U_{\star}^{T}\|^{2}\\
= & \left(\|Y_{L}Y_{R}^{T}-\one\one^{T}\|^{2}+\|Y_{L}Z_{R}\|^{2}+\|Y_{R}Z_{L}\|^{2}+\|Z_{L}^{T}Z_{R}-I_{k}\|^{2}\right)-\|I_{k}\|^{2}\\
= & \left(\|Y_{L}Y_{R}^{T}-\one\one^{T}\|^{2}+\|Z_{L}^{T}Z_{R}\|^{2}\right)+\|Y_{L}Z_{R}\|^{2}+\|Y_{R}Z_{L}\|^{2}-2\inner{Z_{L}}{Z_{R}}\\
\ge & \|Y_{L}Z_{R}\|^{2}+\|Y_{R}Z_{L}\|^{2}-2\inner{Z_{L}}{Z_{R}}\\
\ge & \left(\|Y_{L}Z_{R}\|^{2}-\|Z_{R}\|^{2}\right)+\left(\|Y_{R}Z_{L}\|^{2}-\|Z_{L}\|^{2}\right)\\
= & \inner{Y_{L}^{T}Y_{L}-I}{Z_{R}Z_{R}^{T}}+\inner{Y_{R}^{T}Y_{R}-I}{Z_{L}Z_{L}^{T}}
\end{align*}
The second inequality is because $\|Z_{L}+Z_{R}\|^{2}=\|Z_{L}\|^{2}+\|Z_{R}\|^{2}-2\inner{Z_{L}}{Z_{R}}\ge0$.
In analogy to (\ref{eq:YZbnd}), at $L=R=U_{0}$, we have $Y_{L}=Y_{R}=Y_{0}=\frac{1}{\sqrt{r}}\one_{m}\one_{r}^{T}$,
and hence
\[
\inner{Y_{0}^{T}Y_{0}-I}{Z_{L}Z_{L}^{T}+Z_{R}Z_{R}^{T}}\ge(\frac{m}{r}-1)(\|Z_{L}\|^{2}+\|Z_{R}\|^{2}).
\]
In analogy to (\ref{eq:radiusbnd}), within the neighborhood $\|(L,R)-(U_{0},U_{0})\|\le\rho=\frac{1}{4r\sqrt{m}}$,
we have
\[
\max\{\|Y_{L}^{T}Y_{L}-Y_{0}^{T}Y_{0}\|,\|Y_{R}^{T}Y_{R}-Y_{0}^{T}Y_{0}\|\}\le\frac{1}{r}\quad\text{ for }m,r\ge1.
\]
Combined, the local suboptimality is lower-bounded
\begin{align*}
g(L,R)-g(U_{0},U_{0}) & \ge\inner{Y_{L}^{T}Y_{L}-I}{Z_{R}Z_{R}^{T}}+\inner{Y_{R}^{T}Y_{R}-I}{Z_{L}Z_{L}^{T}}\\
 & \ge(\frac{m}{r}-1-\frac{1}{r})\|Z_{R}\|^{2}+(\frac{m}{r}-1-\frac{1}{r})\|Z_{L}\|^{2}\\
 & \ge0\quad\text{ for integer }m>r\ge1,
\end{align*}
so we conclude that $(L,R)=(U_{0},U_{0})$ is a local minimizer within
the neighborhood.
\end{proof}
Finally, we show that this spurious local minimum persists after adding
the standard balancing regularizer used to eliminate scaling ambiguity
in asymmetric factorizations. 
\begin{lem}
Let $g_{\lambda}(L,R)=g_{0}(L,R)+\frac{\lambda}{2}\|L^{T}L-R^{T}R\|^{2}$
for $g_{0}:\R_{+}^{n\times r}\times\R_{+}^{n\times r}\to\R$ and $\lambda\ge0$.
If $(L_{0},R_{0})$ with $L_{0}^{T}L_{0}=R_{0}^{T}R_{0}$ is a local
minimizer for $g_{0}$, then it is also a local minimizer for $g_{\lambda}$.
\end{lem}

\begin{proof}
By definition, there exists a neighborhood $\|(L,R)-(L_{0},R_{0})\|\le\rho$
such that $g_{0}(L,R)\ge g_{0}(L_{0},R_{0})$. Within that same neighborhood,
$g_{\lambda}(L,R)\ge g_{0}(L,R)\ge g_{0}(L_{0},R_{0})=g_{\lambda}(L_{0},R_{0}).$
\end{proof}

\section{No spurious local minima for $\delta=0$ and $r^{\star}=1$}\label{sec:nospu}

\subsection{Symmetric factorization}

To prove that $f(U)=\frac{1}{2}\|UU^{T}-u_{\star}u_{\star}^{T}\|^{2}$
for $u_{\star}\in\R_{+}^{n}$ has no spurious local minima, we first
show that every first-order critical point $U$ satisfying
\[
U\ge0,\quad\nabla f(U)\ge0,\quad\inner{\nabla f(U)}U=0
\]
is either globally optimal, or it is zero. This is clearly implied
by the lemma below, since $u_{\star}\ge0$ and therefore $u_{\star}u_{\star}^{T}U\ge0$. 
\begin{lem}
\label{lem:nofirst}If $\inner{\nabla f(U)}{u_{\star}u_{\star}^{T}U}\ge0$
and $\inner{\nabla f(U)}U\le0$, then either $U=0$ or $UU^{T}=u_{\star}u_{\star}^{T}$.
\end{lem}

\begin{proof}
Define $v=U^{T}u_{\star}\ge0$, and observe that 
\[
0\;\le\;\half\inner{\nabla f(U)}{u_{\star}v^{T}}=\inner{u_{\star}v^{T}}{(UU^{T}-u_{\star}u_{\star}^{T})U}=\inner{vv^{T}}{U^{T}U}-\inner{v^{T}v}{u_{\star}^{T}u_{\star}},
\]
 and therefore 
\[
\|v\|^{2}\|UU^{T}\|\ge\inner{vv^{T}}{U^{T}U}\ge\inner{v^{T}v}{u_{\star}^{T}u_{\star}}=\|v\|^{2}\|u_{\star}u_{\star}^{T}\|.
\]
Next, we observe that 
\[
\frac{1}{2}\inner{\nabla f(U)}U=\inner U{(UU^{T}-u_{\star}u_{\star}^{T})U}=\|UU^{T}\|^{2}-\|v\|^{2}\le0.
\]
If $U\ne0$, then $\|v\|\ge\|UU^{T}\|\ne0$, and hence $\|UU^{T}\|\ge\|u_{\star}u_{\star}^{T}\|$.
Piecing everything together, we have
\begin{align*}
\|UU^{T}-u_{\star}u_{\star}^{T}\|^{2} & =\|UU^{T}\|^{2}+\|u_{\star}u_{\star}^{T}\|^{2}-2\|v\|^{2}\\
 & \le\|u_{\star}u_{\star}^{T}\|^{2}-\|UU^{T}\|^{2}\le0.
\end{align*}
\end{proof}
Next, we prove that $U=0$ is never a second-order critical point,
except in the trivial case when $u_{\star}=0$. Therefore, we conclude
that all second-order critical points of $f$ are always globally
optimal. 
\begin{lem}
If $u_{\star}\ne0$, then $\inner{\nabla^{2}f(0)[\dot{U}]}{\dot{U}}<0$
for some $\dot{U}\in\Cone(0,\nabla f(0))$.
\end{lem}

\begin{proof}
Note that $\nabla f(0)=0$, and therefore $\Cone(0,\nabla f(0))=\R_{+}^{n\times r}$.
Then, $\dot{U}=u_{\star}\e_{1}^{T}\ge0$ yields $\inner{\nabla^{2}f(0)[\dot{U}]}{\dot{U}}=2\inner{-u_{\star}u_{\star}^{T}}{(u_{\star}\e_{1}^{T})(u_{\star}\e_{1}^{T})^{T}}=-2\|u_{\star}\|^{2}\|\dot{U}\|^{2}.$
\end{proof}

\subsection{Asymmetric factorization}

Now we move to the asymmetric case. To prove for $U=[L;R]\in\R_{+}^{n\times r}$
and $u_{\star}=[a;b]\in\R_{+}^{n}$ with $\|a\|=\|b\|$ that 
\[
g_{\lambda}(U)=\frac{1}{2}\|LR^{T}-ab^{T}\|^{2}+\frac{\lambda}{2}\|L^{T}L-R^{T}R\|^{2}
\]
has no spurious local minima when $\lambda=1/4$, we will again show
that every first-order critical point $U$ is either globally optimal,
or it is zero. We prove this by relating the first-order critical
points of $g_{\lambda}$ to first-order critical points of the symmetric
function $f(U)=\frac{1}{2}\|UU^{T}-u_{\star}u_{\star}^{T}\|^{2}$,
and then evoking \lemref{nofirst} for the symmetric case. 
\begin{lem}
If $U=[L;R]\ge0$ satisfies $\nabla g_{\lambda}(U)\ge0$ and $\inner{\nabla g_{\lambda}(U)}U=0$
for $\lambda=1/4$, then it also satisfies $\inner{\nabla f(U)}{u_{\star}u_{\star}^{T}U}\ge0$
and $\inner{\nabla f(U)}U\le0$. 
\end{lem}

\begin{proof}
Let $J=\diag(I_{n_{1}},-I_{n_{2}})$ so that $JU=[L;-R]$. Define
the error matrix $E=UU^{T}-u_{\star}u_{\star}^{T}$ and partition
it into $2\times2$ blocks so that $E_{12}=E_{21}^{T}=LR^{T}-ab^{T}$.
We observe that, for $\lambda=1/4$, the gradients of $g_{\lambda}$
and $f$ are related
\begin{align*}
\nabla g_{\lambda}(U) & =\left(\begin{bmatrix}0 & E_{12}\\
E_{12}^{T} & 0
\end{bmatrix}+\frac{1}{2}\begin{bmatrix}L\\
-R
\end{bmatrix}\begin{bmatrix}L\\
-R
\end{bmatrix}^{T}\right)\begin{bmatrix}L\\
R
\end{bmatrix}\\
 & =\left(\begin{bmatrix}0 & E_{12}\\
E_{12}^{T} & 0
\end{bmatrix}+\frac{1}{2}\begin{bmatrix}E_{11} & -E_{12}\\
-E_{12}^{T} & E_{22}
\end{bmatrix}+\frac{1}{2}\begin{bmatrix}a\\
-b
\end{bmatrix}\begin{bmatrix}b\\
-b
\end{bmatrix}^{T}\right)\begin{bmatrix}L\\
R
\end{bmatrix}\\
 & =\left(\frac{1}{2}\begin{bmatrix}E_{11} & E_{12}\\
E_{12}^{T} & E_{22}
\end{bmatrix}+\frac{1}{2}\begin{bmatrix}a\\
-b
\end{bmatrix}\begin{bmatrix}b\\
-b
\end{bmatrix}^{T}\right)\begin{bmatrix}L\\
R
\end{bmatrix}\\
 & =\frac{1}{4}\nabla f(U)+\frac{1}{2}Ju_{\star}u_{\star}^{T}JU.
\end{align*}
Regarding the extra term, note that $u_{\star}^{T}Ju_{\star}=0$ by
definition, and therefore
\begin{gather*}
\inner{Ju_{\star}u_{\star}^{T}JU}{u_{\star}u_{\star}^{T}U}=\inner{u_{\star}^{T}Ju_{\star}}{u_{\star}^{T}UU^{T}Ju_{\star}}=0,\\
\inner{Ju_{\star}u_{\star}^{T}JU}U=\|U^{T}Ju_{\star}\|\ge0.
\end{gather*}
Hence, if $\inner{\nabla g_{\lambda}(U)}{u_{\star}u_{\star}^{T}U}\ge0$
then $\inner{\nabla f(U)}{u_{\star}u_{\star}^{T}U}\ge0$, and if $\inner{\nabla g_{\lambda}(U)}U=0$
then $\inner{\nabla f(U)}U\le0$. 
\end{proof}
Again, we prove that $U=0$ is never a second-order critical point
for $g_{\lambda}$, except in the trivial case when $u_{\star}=0$.
Therefore, we conclude that all second-order critical points of $g_{\lambda}$
for $\lambda=1/4$ are always globally optimal. 
\begin{lem}
If $u_{\star}\ne0$, then $\inner{\nabla^{2}g_{\lambda}(0)[\dot{U}]}{\dot{U}}<0$
for some $\dot{U}\in\Cone(0,\nabla g_{\lambda}(0))$.
\end{lem}

\begin{proof}
Note that $\nabla g_{\lambda}(0)=0$, and therefore $\Cone(0,\nabla g_{\lambda}(0))=\R_{+}^{n\times r}$.
Then, $\dot{U}=u_{\star}\e_{1}^{T}\ge0$ yields $\inner{\nabla^{2}g_{\lambda}(0)[\dot{U}]}{\dot{U}}=\inner{\begin{bmatrix}0 & -ab^{T}\\
-ba^{T} & 0
\end{bmatrix}}{(u_{\star}\e_{1}^{T})(u_{\star}\e_{1}^{T})^{T}}=-2\|ab^{T}\|^{2}.$ Finally, since $\|a\|=\|b\|$ holds by assumption, we have $\|ab^{T}\|=\|a\|\|b\|=\frac{1}{2}\|u_{\star}\|^{2}=\frac{1}{2}\|\dot{U}\|^{2}$.
\end{proof}

\section{Conclusions and future work}\label{sec:conclu}

This paper shows that the benign nonconvex geometry long associated
with RIP-based low-rank recovery is not robust to the introduction
of nonnegativity constraints. Although a single fully observed rank-1
case remains benign, this property breaks down under arbitrarily small
partial observation, higher-rank ground truths, and overparameterization,
with explicit spurious local minima appearing at the boundary. These
failures occur in settings that are computationally trivial and easily
solvable by inspection, so they must stem from the geometry of nonnegative
factorization itself rather than algorithmic hardness. As a result,
the standard stability-based explanation that successfully unifies
much of unconstrained low-rank recovery does not extend to the nonnegative
setting. Understanding the empirical success of projected and constrained
methods in nonnegative low-rank recovery applications such as clustering
therefore requires new analytical frameworks that go beyond existing
stability-based benign landscape arguments.

Below, we outline some promising directions of future research.

\paragraph{Small-but-positive elements and apocalypses}

A key component of our construction in \thmref{main} is to induce
a spurious local minimizer $U_{0}$ that contains both exactly zero
elements and also small-but-positive elements. This construction takes
advantage of the phase transition between zero and near-zero entries,
and is closely reminescent of another recent work on so-called apocalypses
in the set of rank-bound matrices~\cite{levin2023finding}. This
connection suggests explicitly thresholding the elements of the factor
matrix $U,L,R$ as a potential safeguard, that might preserve benign
nonconvexity by preventing an algorithm from becoming trapped in shallow
but combinatorial regions. Developing such safeguarded algorithms
or identifying the conditions under which they restore benign nonconvexity
could be a fruitful research direction.

\paragraph{Algorithmic implications and approximate stationarity.}

\noindent In the benign $\delta=0,r^{\star}=1$ regime, we show that
all \emph{exact} constrained second-order stationary points are globally
optimal (and that this fails for $\delta>0$ or $r^{\star}>1$). To
say that standard algorithms such as gradient flow or gradient descent
converge to the global minimum, however, would require the stronger
\emph{strict saddle} property, namely that every $\epsilon$-second-order
stationary point lies $O(\epsilon)$-close to a global optimum. Assuming
this holds, projected gradient descent (and similar methods) can be
equipped with saddle-escape mechanisms to ensure convergence to $\epsilon$-second-order
points~\cite{nouiehed2018convergence,mokhtari2018escaping,nouiehed2020trust}.
On the other hand, in the constrained setting, randomly initialized
projected gradient descent can get stuck at a saddle point with positive
probability~\cite{nouiehed2018convergence}. To fully explain how
algorithms actually behave in the benign $\delta=0,r^{\star}=1$ regime,
therefore, would require a more refined, problem-specific analysis,
which we leave as future work. 

\paragraph{Semidefinite programs for constructing counterexamples.}

In the analysis of benign nonconvexity in unconstrained low-rank matrix
recovery, a key proof technique is the use of semidefinite programming
for constructing counterexamples that admit spurious local minima~\cite{zhang2018much,zhang2019spurious,zhang2024improved}.
This is possible because the second-order conditions in the unconstrained
case reduce to the linear matrix inequalities $\nabla f(U)=0$ and
$\nabla^{2}f(U)\succeq0$. In the nonnegative setting, the second-order
condition requires copositivity $\inner{\nabla^{2}f(U)[\dot{U}]}{\dot{U}}\ge0$
over the critical cone $\dot{U}\in\Cone(U,\nabla f(U))$, which in
turn is \emph{coupled back to the candidate solution itself}, through
both $U$ and $\nabla f(U)$. SDP-based constructions do not readily
extend, and this is why our counterexample in \thmref{main} was manually
constructed by hand. Nevertheless, a more systematic approach to counterexamples
could illuminate deeper principles behind why the benign landscape
disappears under nonnegativity constraints. 

\paragraph{Squared-variable reparameterization}

A natural possibility suggested by a reviewer is to eliminate the
explicit constraint $U\ge0$ by reparameterizing $U=V\odot V$ (entrywise),
so that the optimization over $V\in\R^{n\times r}$ becomes unconstrained.
From the perspective of analyzing benign landscapes, this reparameterization
would restore access to the SDP machinery developed for unconstrained
problems, as mentioned above. At the same time, it is important to
note that the landscape of $\varphi(V)=f(V\odot V)$ can differ qualitatively
from the landscape of $f(U)$ over $U\ge0$. In particular, benignness
(or non-benignness) need not transfer between the two parametrizations.
Nevertheless, studying basic benign-landscape questions for $\varphi(V)$,
and attempting SDP-based counterexample searches in $V$, is a promising
direction, as it could yield useful insight into whether alternative
parameterizations can mitigate bad landscapes induced by the constraint
$U\ge0$.

\subsubsection*{Acknowledgments}

I am grateful to Xiaohui Chen, Salar Fattahi, Sabrina Zielinski, and
three reviewers for helpful feedback and insightful suggestions that
significantly strengthened the paper. 

\subsubsection*{Funding}

Financial support for this work was provided by NSF CAREER Award ECCS-2047462
and ONR Award N00014-24-1-2671.

\subsubsection*{Data Availability}

Data sharing not applicable to this article as no datasets were generated
or analysed during the current study.

\subsection*{Declarations}

\subsubsection*{Conflict of interest}

The author has no relevant financial or non-financial interests to
disclose.

\bibliographystyle{plain}
\bibliography{proof_half}

\end{document}